\documentclass[12pt,oneside]{amsart}
\usepackage{amssymb,amsfonts,amsmath}
\usepackage[]{graphicx}
\usepackage{color}

\newcommand{\ol}[1]{{\overline{#1}}}

\newcommand{\bigpar}{\par\quad\newline\noindent}
\newcommand{\dint}[1]{\,\mathrm{d}#1}

\newcommand{\iu}{\mathtt{i}}

\newcommand{\fspace}[1]{{\mathsf{#1}}}
\newcommand{\fspaceL}{\fspace{L}}
\newcommand{\fspaceH}{\fspace{H}}

\newcommand{\Rset}{{\mathbb{R}}}
\newcommand{\Zset}{{\mathbb{Z}}}

\newcommand{\cointerval}[2]{[#1,\,#2)}%
\newcommand{\ccinterval}[2]{[#1,\,#2]}%

\newcommand{\DO}[1]{{O\at{#1}}}

\newcommand{\pair}[2]{{\left({#1},\,{#2}\right)}}

\newcommand{\bpair}[2]{{\big({#1},\,{#2}\big)}}

\newcommand{\at}[1]{{\left({#1}\right)}}
\newcommand{\nat}[1]{(#1)}
\newcommand{\bat}[1]{{\big(#1\big)}}
\newcommand{\Bat}[1]{{\Big(#1\Big)}}

\newcommand{\triple}[3]{{\left({#1},\,{#2},\,{#3}\right)}}

\newcommand{\norm}[1]{\|{#1}\|}
\newcommand{\bnorm}[1]{\big\|{#1}\big\|}

\newcommand{\abs}[1]{\left|{#1}\right|}
\newcommand{\babs}[1]{\big|{#1}\big|}

\newcommand{\eps}{{\varepsilon}}

\newcommand{\ka}{{\kappa}}

\newcommand{\si}{{\sigma}}

\newtheorem{theorem}{Theorem}[]

\newtheorem{lemma}         [theorem]{Lemma}
\newtheorem{proposition}   [theorem]{Proposition}

\newtheorem*{result*}{Main results}
\newtheorem*{problem*}{Open problems}

\begin{document}

\title{Uniqueness of solitary waves
in the high-energy limit of  FPU-type chains}
\date{\today}
\author{Michael Herrmann}
  \address{Institut f\"ur Numerische und Angewandte Mathematik,
     Universit\"at M\"unster}
     \email{michael.herrmann@uni-muenster.de}
     
  \author{Karsten Matthies}
\address{Department of Mathematical Sciences, University of Bath} \email{k.matthies@bath.ac.uk}

\begin{abstract}Recent asymptotic results in \cite{HM15} provided detailed information   on the shape of    solitary high-energy travelling waves   in   FPU atomic chains. In this note we use and extend
the methods to understand the linearisation of the travelling wave equation. We show that there are not any other zero eigenvalues than those created by the translation symmetry   and this implies a local uniqueness result. The key \mbox{argument} in our asymptotic analysis is to replace the linear advance-delay-differential equation for the eigenfunctions by an approximate ODE.\end{abstract}

\maketitle

\section{Introduction}

We study an aspect of coherent motion within  a spatially one-dimensional lattice with nearest-neighbor interactions in the form of Fermi-Pasta-Ulam or FPU-type chains given by
\begin{align}
\label{Eqn:FPU}
\ddot u_j\at{t}= \Phi^\prime\bat{u_{j+1}\at{t}-u_{j}\at{t}}-
\Phi^\prime\bat{u_{j}\at{t}-u_{j-1}\at{t}}, \qquad j \in \Zset.
\end{align}
We are interested in  solitary travelling waves, which are solutions of \eqref{Eqn:FPU},
given for
 positive wave-speed parameter $\si$ by a distance profile $R$ and a velocity profile $V$ such that
\begin{align}
\label{Eqn:TW.Diff}
R^\prime\at{x}
=
V\at{x+1/2}-
V\at{x-1/2}\,,\quad
\si \, V^\prime\at{x}
=
\Phi^\prime\bat{R\at{x+1/2}}-
\Phi^\prime\bat{R\at{x-1/2}}
\end{align}
is satisfied for all $x\in\Rset$. The scalar function $\Phi$ is the nonlinear interaction potential and the position $u_j\at{t}$ of particle $j$ can be obtained by $u_j\at{t}=U\at{j-\sqrt{\si}\,t}$,   where $U$ denotes the primitive of $V$.
\par
  In the literature there exist many results on the existence of different types of travelling waves -- see for instance \cite{FW96,FV99,Pan05,SZ09,IJ05} -- but almost nothing is known about the uniqueness for fixed wave-speed or
their dynamical stability with respect to \eqref{Eqn:FPU}.
The only exceptions are the completely integrable Toda chain (see \cite{Tes01} for an overview) and the KdV limit of near-sonic waves with small energy which have been studied rigorously in \cite{FP99,FP02,FP3,FP4}.
\par
Another asymptotic regime is related to high-energy waves in chains with rapidly increasing or singular potential; we refer to
\cite{FM02,Tre04,Her10,H17} for FPU-type chains and to \cite{FSD12,TV14,AKJSG15} for similar solutions in other
models. In \cite{HM15} the authors   provide a detailed asymptotic analysis for the high-energy limit for potentials with
sufficiently strong singularity and   derive
explicit leading \mbox{order} formula for $\si$  as well as the next-to-leading \mbox{order} corrections to the \mbox{asymptotic} \mbox{profile} functions.
In this note we   apply similiar techniques to the linearisation of (\ref{Eqn:TW.Diff}) and sketch how the local uniqueness of  solitary high-energy waves can be estab\-lished by   an implicit function argument.
In the final section \ref{sect:dis}, we set the results into the wider context of stable coherent motion for FPU lattices.

%
%
\section{The high-energy limit for singular potentials}
%
%
%
  As in \cite{HM15}   we restrict our considerations to the example potential
\begin{align}
\label{Eqn:Pot}
\Phi\at{r} = \frac{1}{m\at{m+1}}\at{\frac{1}{\at{1-r}^m}-m\,r-1}\qquad \text{with}\quad m\in\Rset\quad\text{and}\quad m>1\,,
\end{align}
which satisfies $\Phi\at{0}=\Phi^\prime\at{0}=0$ and $\Phi^{\prime\prime}\at{0}=1$.   This potential is convex, well-defined for $r\leq1$, and singular as $r\nearrow1$. Moreover, it resembles -- up to a reflection in $r$  -- the classical Lennart-Jones potential, for which the analysis holds with minor modifications.
\par
  The subsequent analysis concerns a special family of solitary waves that has been introduced in \cite{HM15}; similar families have been constructed in \cite{FM02,Tre04,Her10}.
\begin{proposition}[family of solitary waves and its high-energy limit]
\label{Ass:Waves}
There exists a family of solitary waves $\bat{\triple{V_\delta}{R_\delta}{\si_\delta}}_{0<\delta<1}$  with the following properties:
\begin{enumerate}
\item
$V_\delta$ and $R_\delta$ belong to $\fspaceL^2\at\Rset\cap\fspace{BC}^1\at\Rset$ and are nonnegative and even. They are also unimodal, i.e. increasing and decreasing for $x<0$ and $x>0$, respectively.
\item
$V_\delta$ is
normalized by $\norm{V_\delta}_2=1-\delta$ and
$R_\delta$ takes values in $\cointerval{0}{1}$.
\end{enumerate}
Moreover, the potential energy explodes in the sense of
$p_\delta:=\int_\Rset \Phi\bat{R_\delta\at{x}}\dint{x}\to \infty$ as $\delta\to0$.
\end{proposition}
The asymptotic results from \cite{HM15} can be summarized as follows, where the small quantities
\begin{align*}
\eps_\delta := 1-R_\delta\at{0}  \qquad\text{and}\qquad
\mu_\delta:=\sqrt{\si_\delta\,\eps_\delta^{m+2}}
\end{align*}
  measure the inverse impact of the singularity   and
determine   the length scale for the leading order corrections to the asymptotic profile functions,   respectively.
\par
\begin{proposition}[global approximation in the high-energy limit]
\notag
The formulas
\begin{equation*}
\hat{R}_\eps\at{x}:=\left\{
\begin{array}{lcl}
1-\eps-\eps\,\bar{S}\at{\displaystyle\frac{\abs{x}}{\hat{\mu}_\eps}}&&\mbox{for $0\leq\abs{x}<\frac12$}\\
\eps\,\bar{T}\at{\displaystyle\frac{1-\abs{x}}{\hat{\mu}_\eps }}&&\mbox{for $\frac12\leq\abs{x}<\frac32$}\\
0&&\mbox{else}
\end{array}\right.\
\end{equation*}
and
\begin{equation*}
\hat{V}_\eps\at{x}:=\frac{\eps}{\hat{\mu}_\eps}\left\{
\begin{array}{lcl}
\bar{W}\at{\displaystyle\frac{\frac12-\abs{x}}{\hat{\mu}_\eps }}&&\mbox{for $0\leq\abs{x}<1$}\\
0&&\mbox{else}
\end{array}\right.
\end{equation*}
with
\begin{equation*}
\hat{\mu}_\eps := \frac{\ol{\mu}\,\eps}{1+\eps\,\at{\ol\ka-1}}\,\qquad \hat{\si}_\eps := \eps^{-m-2}\,\hat{\mu}_\eps^2
\end{equation*}
approximate the solitary waves from Proposition \ref{Ass:Waves} in the sense of
\begin{equation*}
\bnorm{R_\delta-\hat{R}_{\eps_\delta}}_q +\bnorm{V_\delta-\hat{V}_{\eps_\delta}}_q  + \eps_\delta^{-1}\babs{\mu_{\delta}-\hat{\mu}_{\eps_\delta}}+\eps_\delta^m \babs{\si_{\delta}-\hat{\si}_{\eps_\delta}}=\DO{\eps_\delta^m}=\DO{\delta^m}
\end{equation*}
for any $q\in\ccinterval{1}{\infty}$. Here,
$\bar{S}$ solves the ODE initial-value problem
\begin{equation}
\label{Eqn:LimitIVP}
\bar{S}^{\prime\prime}\at{\bar{x}}=
\frac{2}{m+1}\cdot\frac{1}{\at{1+\bar{S}\at{\bar{x}}}^{m+1}}\,,\qquad \bar{S}\at{0}=\bar{S}^\prime\at{0}=0
\end{equation}
and we have $\ol{\mu}:=\frac{2}{\sqrt{m\at{m+1}}}$ and $\ol{\ka}:=\int_{0}^\infty \tilde{x}\,\bar{S}^{\prime\prime}\at{\bar{x}}\dint{\bar{x}}$ as well as
$\bar{W}\at{\bar{x}}:=\mbox{$\frac12$}\at{\bar{S}^\prime\at{\bar{x}}+\ol{\mu}}$ and $\bar{T}\at{\bar{x}}:=\mbox{$\frac12$}\at{\bar{S}\at{\bar{x}}+\ol{\mu}\,\bar{x}+\ol{\ka}}$.
\end{proposition} %
\bigpar
\begin{figure}[t!] %
\centering{ %
\includegraphics[width=0.98\textwidth]{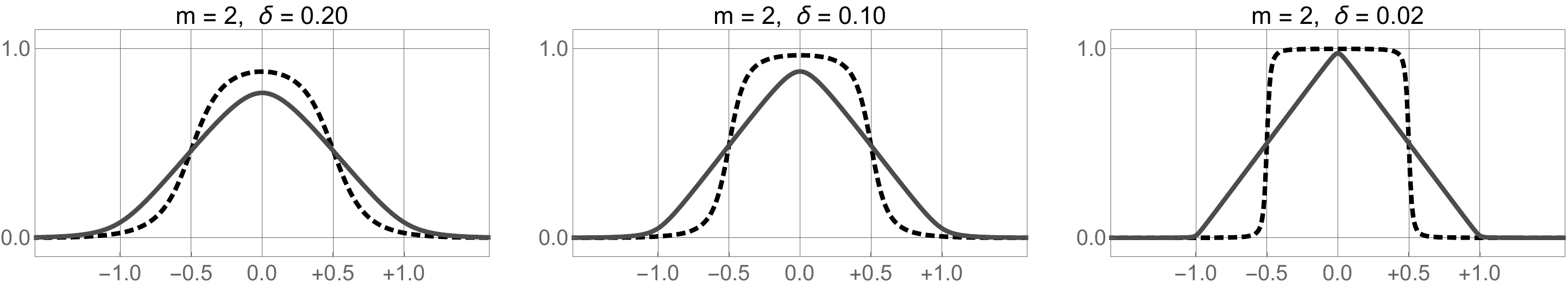} %
} %
\caption{Numerical results for high energy waves with $m=2$: The velocity profile $V_\delta$ (black, dashed) approaches as $\delta\to0$ the indicator function $V_0$ while the distance profiles $R_\delta$ (gray, solid) converges to the tent map $R_0$.%
} \label{FigNum} 
\end{figure} %
In this note we   establish   a local uniqueness result for the solitary waves from Proposition~\ref{Ass:Waves}.
\begin{theorem}  \label{ass:basic}
Suppose that $\delta_0>0$ is sufficiently small, then the solitary waves $R_\delta,V_\delta$ for given $\sigma_\delta$ are
locally unique for $0<\delta<\delta_0$. More precisely, there exists $c_0>0$ such that there are no other non-negative, even, and unimodal solutions $(R,V)$ of
\eqref{Eqn:TW.Diff} for fixed $\sigma_\delta$ with $|(R,V)-(R_\delta,V_\delta)|_{\fspaceL^2}\leq c_0$.

Furthermore the family $R,V$ depends continuously on the wave parameter $\sigma$.

\end{theorem}

The proof is based on an implicit function argument applied to the nonlinear travelling wave operator
\begin{equation}\label{eq:F}
    \mathcal{F}(R,V,\si_\delta)=\left(\begin{array}{l} \frac{\partial}{\partial x }R(.)-  \nabla^\pm_{1/2} V(.)\\
 \frac{\partial}{\partial x }V(.)- \frac 1 {\si_\delta} \nabla^\pm_{1/2}
\left(\Phi'(R(.))  \right) \end{array} \right),
\end{equation}
  where the main challenge is to control the kernel of its linearisation.
%
%
\section{Linearisation}
%

  The linearisation of (\ref{eq:F}) around a travelling wave
$\pair{R_\delta}{V_\delta}$ with speed $\si_\delta$ reads
\begin{align}
\label{Eqn:pullback}
L_\delta \left(\begin{array}{l}
S(.)\\ W(.)
\end{array} \right) &=\left(\begin{array}{l} \frac{\partial}{\partial x }S(.)-  \nabla^\pm_{1/2} W(.)\\
 \frac{\partial}{\partial x }W(.)- \frac 1 {\si_\delta} \nabla^\pm_{1/2}
\left(\Phi''(R_\delta(.)) S(.) \right) \end{array} \right)
\end{align}
  with $\nabla^\pm_{1/2}$ being the standard centered-difference operator with spacing $1/2$ . We consider $L_\delta$ as an operator on
the weighted Sobolev space
\begin{align}
\notag 
\fspaceL^2_a:=\{(S,W): \Rset \to \Rset^2 : \exp(ax) (S(x),W(x)) \in \fspaceL^2(\Rset,\Rset^2)\},
\end{align}
  which is for given parameter $a>0$ defined on the dense subspace
\begin{align}
\notag 
\fspaceH^1_{  a  }:=\{(S,W): \Rset \to \Rset^2 : \exp(  a  x) (S(x),W(x)) \in \fspaceH^1(\Rset,\Rset^2)\}.
\end{align}

  The first important observation is that the shift symmetry of \eqref{eq:F} implies
that $L_\delta$ has at least one kernel function.

\begin{lemma}
  Let $a>0$ be given, $\delta>0$ be sufficiently small, and $(R_\delta,V_\delta)$ be a travelling wave. Then
\begin{equation}\label{eq:EV}
 \bpair{S_{1,\delta}}{W_{1,\delta}}:=\pair{\frac{d R_\delta}{dx}}{\frac{d V_\delta}{dx}}
\end{equation}
is in the kernel of $L_\delta$ and belongs to $\fspaceH^1_a \cap \fspaceH^1_{-a}$.
\end{lemma}
\begin{proof} The identity $L_\delta \bpair{S_{1,\delta}}{W_{1,\delta}}=0$ is obtained by differentiating \eqref{Eqn:TW.Diff} with \mbox{respect} to  $x$. The decay properties follow from ideas in \cite{HR10}  as in \cite[Thm.~10]{HM15}.
\end{proof}

  Our main asymptotic result can be formulated as follows and will be proven in several steps.
\begin{proposition}\label{prop:ker}
There exists $\delta_0>0$ such that
\[\mathrm{ker}\, L_\delta= \mathrm{span}  \,\left\{\bpair{S_{1,\delta}}{W_{1,\delta}}\right\}\]
  holds for all $ 0<\delta<\delta_0$.
\end{proposition}
\subsection{Prelimenaries}
  In what follows we denote the wave speed by $c=\sqrt{\si}$.
\begin{lemma}\label{lem:nabl}\quad
\begin{enumerate}
  \item[(a)]  The operator $\nabla^\pm_{1/2}$ is invertible on $\fspaceL^2_a$ for $a>0$.
  \item[(b)]The operator $L_\delta: \fspaceH^1_a \to \fspaceL^2_a$ is Fredholm for $0<a<a_c$, where $a_c>0$ is   the uniquely determined by   $\sinh(a_c/2)/(a_c/2))=c$.
 \end{enumerate}
\end{lemma}
\begin{proof}
 Part (a)   follows by Fourier arguments since $\nabla^\pm_{1/2}$ acts on $\fspaceL^2_a$ as a weighted difference operator . For part (b), the essential spectrum can be calculated explicitly as in \cite[Lem. 4.2]{FP3}. For any $a \in \Rset$, the essential spectrum of $L_\delta$ in $\fspaceL^2_a$ is given by  the following union of two curves:
\begin{equation}
\notag 
\{\lambda : \lambda=\mathcal{P}_+(i k-a) \mbox{ for some } k \in \Rset\} \cup \{\lambda : \lambda=\mathcal{P}_-(i k-a) \mbox{ for some } k \in \Rset\},
\end{equation}
with $\mathcal{P}_\pm(\mu)=\mu\pm 2\frac{1}{\sqrt{\si_\delta}} \sinh(\mu/2)$. In particular,
\[\max\{ \mbox{Re}  \lambda : \lambda \in \sigma_{ess}(L_\delta)\}=-a+ \frac{2}{c} \left|\sinh\left(\frac a 2\right)\right|=:-b_*(c,a)<0,\]
  so   the  essential spectrum does not intersect the closed right complex half plane and hence $0$ if and only if $c>1$ and  $0<a<a_c$, where $a_c>0$ is the solution of the given transcendental equation and   increases with   $c$. As $0$ is not in the essential spectrum $L_\delta$, the operator itself is Fredholm.
\end{proof}
\subsection{Rescaling}
We next transform \eqref{Eqn:pullback} into a second-order advance-delay-differential equation. \mbox{Letting}
$S_\delta(x)= \exp(-a x) G_\delta(x)$ with $G_\delta \in \fspaceL^2$   we   express the linearised equation as
\begin{equation}\label{eq:2nd}
\left( \frac{d}{dx} -a \right)^2 G_\delta= \Delta_{1,-a}  \frac{1}{\si_\delta} Q_\delta G_\delta\,,
\end{equation}
where the    transformed discrete Laplacian is given by
\begin{align}
\label{Def.Mod.Lapl}
\Delta_{1,-a} F (x) = \exp(-a) F(x+1)+\exp(+a) F(x-1)-2 F(x)\,.
\end{align}
Any solution $G_\delta$ to \eqref{eq:2nd} gives immediately a corresponding $S_\delta$ and then due to the invertibility of $\nabla^{\pm}_{1/2}$ on
$\fspaceL^2_a$ also $W_\delta$ to obtain a solution of \eqref{Eqn:pullback}.

  The key asymptotic observation for the high-energy limit $\delta\to0$ is that the advance-delay-differential equation \eqref{eq:2nd} implies an effective ODE for both $G_\delta$ and $S_\delta$ in the vicinity of $x=0$ (`tip of the tent' in Fig.\ref{FigNum}). We therefore rescale the profile $G_\delta$ according to
\begin{equation}
\notag 
    x=\delta \tilde x, \quad \tilde G_\delta (\tilde x)= G_\delta (\delta \tilde x), \quad \tilde Q_\delta(\tilde x)=  \frac{\delta^2}{\si_\delta} Q_\delta (\delta \tilde x), \quad \frac{d}{d\tilde x}= \frac 1 \delta \frac{d}{d x}\,.
\end{equation}
  With respect to the   new coordinates,  \eqref{eq:2nd}   becomes
\begin{equation}\label{eq:2ndtil}
    \left( \frac{d}{  d\tilde{x} } -\delta a \right)^2 \tilde G_\delta= \Delta_{1/\delta,-a}\bat{\tilde Q_\delta\tilde G_\delta} ,
\end{equation}
  where the operator $\Delta_{1/\delta,-a}$ is defined analogously to \eqref{Def.Mod.Lapl}   with spacing $\delta^{-1}$.   Moreover, the Green's function of the differential operator on the left hand side is given by
\begin{equation}\label{eq:greenHtil}
    \tilde H_\delta(\tilde x)= -\tilde x \exp(\delta a \tilde x) \chi_{(-\infty,0)}(\tilde x)
\end{equation}
and   the corresponding convolution operator has the following properties.
\begin{lemma} \label{lem:1}There exists a   constant $C>0$ which depends on the parameter $a$ but not on $\delta$   such that for all $\tilde F \in \fspaceL^2$ we have
\begin{enumerate}
\renewcommand{\labelenumi}{(\roman{enumi})}
\item $\tilde H_\delta \ast (\Delta_{1/\delta,-a} \tilde F)=(\Delta_{1/\delta,-a}\tilde H_\delta) \ast  \tilde F$,
\item   $\delta^2 \norm{\tilde H_\delta \ast \tilde F }_2+\delta \norm{\at{\tilde H_\delta \ast \tilde F}^\prime }_2+\norm{\at{\tilde H_\delta \ast \tilde F}^{\prime\prime}}_2 \leq C \norm{\tilde F}_2$,
\item    $\delta^{1/2}\norm{\at{\tilde H_\delta \ast \tilde F}^\prime }_\infty \leq C \norm{\tilde{F}}_2$ and
$\norm{\at{\tilde H_\delta \ast \tilde F}^\prime }_\infty\leq C\norm{\tilde{F}}_1$.
             \end{enumerate}
\end{lemma}
\begin{proof} Part (i) follows immediately from the Fourier representation of $\Delta_{1/\delta,-a}$ and $\tilde H_\delta$.   In particular, the symbol of $\tilde H_\delta$ is $
h_\delta\nat{\tilde{k}}:=\at{\iu \tilde{k}-\delta a}^{-2}$, so part (ii) is a direct consequence of Parseval's inequality. We finally observe that Young's inequality
yields
\[ \bnorm{ \at{\tilde H_\delta \ast \tilde F}^\prime}_\infty=\bnorm{ \tilde H_\delta' \ast \tilde F}_\infty \leq \bnorm{ \tilde H_\delta'\|_q\|\tilde F}_p \quad \mbox{ with } \quad \frac 1 q +\frac 1 p =1,\]
and hence part (iii) via $\|\tilde H_\delta'\|_2 \leq C \delta^{-1/2}$ for $p=2$ and $\|\tilde H_\delta'\|_\infty \leq C$ for $p=1$.
\end{proof}

Our asymptotic analysis strongly relies on the following characterisation of $\tilde{Q}_\delta$.

\begin{proposition}[properties of the coefficient function]\label{Prop.Coeff}\quad
\begin{enumerate}
\item We have
\begin{align*}
\tilde Q_\delta(\tilde x)=\tilde P(\tilde x) +\delta^{m+2} \tilde Z_\delta(\tilde x )
\end{align*}
where
$\tilde{P}$ is even, decays as $\tilde{x}^{-m-2}$ as $\tilde{x}\to\infty$, and does not depend on $\delta$, while the perturbation $\tilde{Z}_\delta$ is uniformly bounded in $\fspaceL^\infty$.
\item The solution space of the ODE
\begin{align}
\label{eqn:linODE}
\tilde{T}^{\prime\prime}=-2\tilde{P}\tilde{T}
\end{align}
is spanned by an even function $\tilde{T}_e$ and an odd function $\tilde{T}_o$, which can be normalized by
\begin{align*}
\tilde{T}_e^\prime\at{\tilde{x}}\quad\xrightarrow{\;\;\tilde{x}\to+\infty\;\;}\quad 1\,,\qquad
\tilde{T}_o\at{\tilde{x}}\quad\xrightarrow{\;\;\tilde{x}\to+\infty\;\;}\quad 1
\end{align*}
and satisfy
\begin{align*}
\sup_{\tilde{x}\in\Rset}\Bat{\babs{\tilde{T}_e^\prime\at{\tilde{x}}\tilde{x}-\tilde{T}_e\at{\tilde{x}}}+\babs{\at{\tilde{T}_e^\prime\at{\tilde{x}}-1}\tilde{x}^{m}}+\babs{\tilde{T}_o^\prime\at{\tilde{x}}\tilde{x}^{m}}}\leq C
\end{align*}
for some constant $C$ depending on $m$.
\end{enumerate}

\end{proposition}
\begin{proof}
We refer to \cite{HM15} for the details but mention that
the coefficient function $\tilde{P}$ has been constructed from the solution of the nonlinear ODE initial-value problem \eqref{Eqn:LimitIVP}. In a nutshell, we have $\tilde{P}:=\Psi^{\prime\prime}\at{\tilde{R}_*}$, where
$\tilde{R}_*$ is the even and asymptotically affine solution to
\begin{align*}
\tilde{R}_*^{\prime\prime}=-2\Psi^\prime\at{\tilde{R}_*}\,,\qquad \tilde{R}_*\at{0}=0\,,\qquad \tilde{R}_*\at{0}=1
\end{align*}
with $\Psi^{\prime}\at{r}\sim \at{r}^{-m-1}$ for large $\abs{r}$.
In particular,
$\tilde{P}$ has the \emph{non-generic property} that the odd solution to the linear ODE \eqref{eqn:linODE} is asymptotically constant as it is given by $\tilde{T}_o=c\tilde{R}_*^\prime$ for some constant $c$. The remaining assertions on \eqref{eqn:linODE}
follow from standard ODE arguments and the estimates for $\tilde{Z}_\delta$ are provided by
an asymptotic analysis of the nonlinear advance-delay-differential equation \eqref{Eqn:TW.Diff}.
\end{proof}

Using \eqref{eq:greenHtil} and Proposition \ref{Prop.Coeff} we can finally transform
\eqref{eq:2ndtil} into the fixed point problem
\begin{equation}\label{eq:G}
    \tilde G_\delta= \tilde H_\delta \ast \Bat{\Delta_{1/\delta,-a} \left(\tilde P \tilde G_\delta+\delta^{m+2} \tilde Z_\delta\tilde G_\delta\right)}
\end{equation}
and are now in the position to characterize the kernel of $L_\delta$ by
identifying the aforementioned asymptotic ODE.
\subsection{Sketch of the proof of Proposition \ref{prop:ker}}
  In this section we fix $a>0$, consider families $\at{\tilde{G}_\delta}_{0<\delta<1}\subset\fspaceL^2$ of solutions to \eqref{eq:G}, and show that $\tilde{G}_\delta$ is -- up to normalisation factors and small error terms -- uniquely determined.
\bigpar
{\bf Compactness:}   %
Bootstrapping shows that $\tilde{G}_\delta$ is smooth, and without loss of generality we normalise $\tilde G_\delta$ by
\begin{equation}\label{eq:Gnormal}
   \babs{\tilde G_\delta(0)}+\| \tilde P \tilde G_\delta\|_1 + \|\tilde P \tilde G_\delta\|_2=1.
\end{equation}
In view of Lemma \ref{lem:1} -- and thanks to \eqref{eq:G}, $\norm{\tilde{Z}_\delta}_\infty\leq C$, and the uniform $\fspaceL^p$-continuity of the operator $\Delta_{1/\delta,-a}$ -- we estimate
\begin{align*}
\norm{\tilde{G}_\delta}_2\leq C\delta^{-2}\norm{\tilde{P}\tilde{G}_\delta}_2+C\delta^m\norm{\tilde{G}_\delta}_2
\end{align*}
and obtain  $\norm{\tilde{G}_\delta}_2\leq C\delta^{-2}$ for all sufficiently small $\delta>0$. Moreover, using
Lemma~\ref{lem:1} again as well as $m>1$ we find
\begin{align}
\notag 
\norm{\tilde{G}_\delta^\prime}_\infty\leq C\at{\bnorm{\tilde{P}\tilde{G}_\delta}_1+\delta^{m+3/2}\norm{\tilde{G}_\delta}_2}\leq C
\end{align}
and
\begin{align*}
\norm{\tilde{G}_\delta^{\prime\prime}}_2\leq C\Bat{\bnorm{\tilde{P}\tilde{G}_\delta}_2+\delta^{m+2}\norm{\tilde{G}_\delta}_2}\leq C\,,
\end{align*}
which in turn give rise to uniform Lipschitz and  H\"older estimates for $\tilde{G}_\delta$ and $\tilde{G}_\delta^\prime$, respectively. By the Arzel\`a-Ascoli theorem we can therefore extract a (not relabeled) subsequence such that $\tilde G_\delta$ converges in $\fspace{BC}^1_\mathrm{loc}$ to a  limit function $\tilde G_0$. The bounds for $\tilde G_\delta(0)$ and $\norm{\tilde{G}_\delta^\prime}_\infty$ ensure
\begin{align}
\label{Eqn:AffBounds}
|\tilde G_\delta(\tilde x)|  \leq 1 + C|\tilde x|
\end{align}
and hence
  \[ \| \tilde P \tilde G_\delta - \tilde P \tilde G_0\|_1\;\; +\;\; \|\tilde P \tilde G_\delta- \tilde P \tilde G_0\|_2 \quad \xrightarrow{\;\;\delta \to 0\;\;}\quad 0\]
by dominated convergence and due to the tightness of $\tilde P$. In particular, the limit $\tilde G_0$ does not vanish as it also satisfies the normalisation condition \eqref{eq:Gnormal}.
\bigpar
{\bf Asymptotic ODE:}   We next study the functions $\tilde{S}_\delta$ with
\begin{align*}
\tilde{S}_\delta\at{\tilde{x}}:=\exp\at{-a\delta \tilde{x}}\tilde{G}_\delta\at{\tilde{x}}=S_\delta\at{\delta\tilde{x}}\,,
\end{align*}
which also converge in $\fspace{BC}^1_\mathrm{loc}$ to the nontrivial limit $\tilde{S}_0=\tilde{G}_0$ and satisfy the advance-delay-differential equation
\begin{align}
\label{Eqn:ADDScaled}
\tilde{S}_\delta^{\prime\prime}=\Delta_{1/\delta,0}\Bat{\at{\tilde{P}+\delta^{m+2}\tilde{Z}_\delta}\tilde{S}_\delta}
\end{align}
thanks to \eqref{eq:2ndtil}, where $\Delta_{1/\delta,0}$ abbreviates the discrete Laplacian with spacing $1/\delta$ and standard weights. Combining
\eqref{Eqn:ADDScaled} with the decay of $\tilde{P}$,  the uniform bounds for $\tilde{Z}_\delta$, as well as the affine bound for $\tilde{G}_\delta$ from \eqref{Eqn:AffBounds} we obtain
\begin{align*}
\abs{\tilde{S}_\delta\at{\tilde{x}}}+
\abs{
\tilde{S}^\prime_\delta\at{\tilde{x}}}\quad\xrightarrow{\;\;\tilde{x}\to+\infty\;\;}\quad 0
\end{align*}
as well as
\begin{align*}
\abs{\tilde{S}_\delta^{\prime\prime}\at{\tilde{x}}}\leq C\frac{\exp\at{-\delta a\tilde{x}}\bat{\tilde{x}+\delta^{-1}}}{\bat{\tilde{x}-\delta^{-1}}^{m+2}}
\qquad \text{for}\quad \tilde{x}\geq \tfrac{3}{2}\delta^{-1}
\end{align*}
and hence
\begin{align}
\label{Eqn:AuxEstimates}
\abs{\tilde{S}_\delta\at{\tfrac{3}{2}\delta^{-1}}}=\DO{\delta^{m-1}}\,,\qquad
\abs{\tilde{S}^\prime_\delta\at{\tfrac{3}{2}\delta^{-1}}}=\DO{\delta^{m}}
\end{align}
after integration over $\tilde{x}\geq \tfrac{3}{2}\delta^{-1}$. Using the pointwise estimates and the decay of $\tilde{P}$ we further verify
\begin{align}
\label{AsympODE1}
\tilde{S}_\delta^{\prime\prime}\at{\tilde{x}}=-2\tilde{P}\at{\tilde{x}}\tilde{S}_\delta
\at{\tilde{x}} +\tilde{E}_{0,\delta}\at{\tilde{x}}\qquad \text{for}\quad \tilde{x}\in \tilde{I}_\delta:=\left[-\tfrac{1}{2}\delta^{-1},\,+\tfrac{1}{2}\delta^{-1}\right]\,
\end{align}
as well as
\begin{align}
\label{AsympODE2}
\tilde{S}_\delta^{\prime\prime}\at{\tilde{x}+\delta^{-1}}=\tilde{P}\at{\tilde{x}}\tilde{S}_\delta
\at{\tilde{x}} +\tilde{E}_{+,\delta}\at{\tilde{x}}\qquad \text{for}\quad \tilde{x}\in \tilde{I}_\delta,
\end{align}
where the error terms are pointwise of order $\DO{\delta^m}$ and satisfy
\begin{align}
\label{eqn:errorEst}
\int_{\tilde{I}_\delta} \abs{\tilde{x}}^i\bat{\abs{\tilde{E}_{0,\delta}\at{\tilde{x}}}+\abs{\tilde{E}_{+,\delta}\at{\tilde{x}}}}\dint\tilde{x}=\DO{\delta^{m-i}}\qquad \text{for}\quad i\in\{0,1\}\,.
\end{align}
In other words, we can replace the nonlocal equation \eqref{Eqn:ADDScaled} on the interval $\tilde{I}_\delta$ by an asymptotic ODE since both the advance and the delay terms on the right hand side are small, while on the shifted interval $\tilde{I}_\delta+\delta^{-1}$ the main contribution stems from the delay term. (On $\tilde{I}_\delta-\delta^{-1}$, the advance term is the most relevant one.)
\bigpar
\begin{figure}[t!] %
\centering{ %
\includegraphics[width=0.5\textwidth]{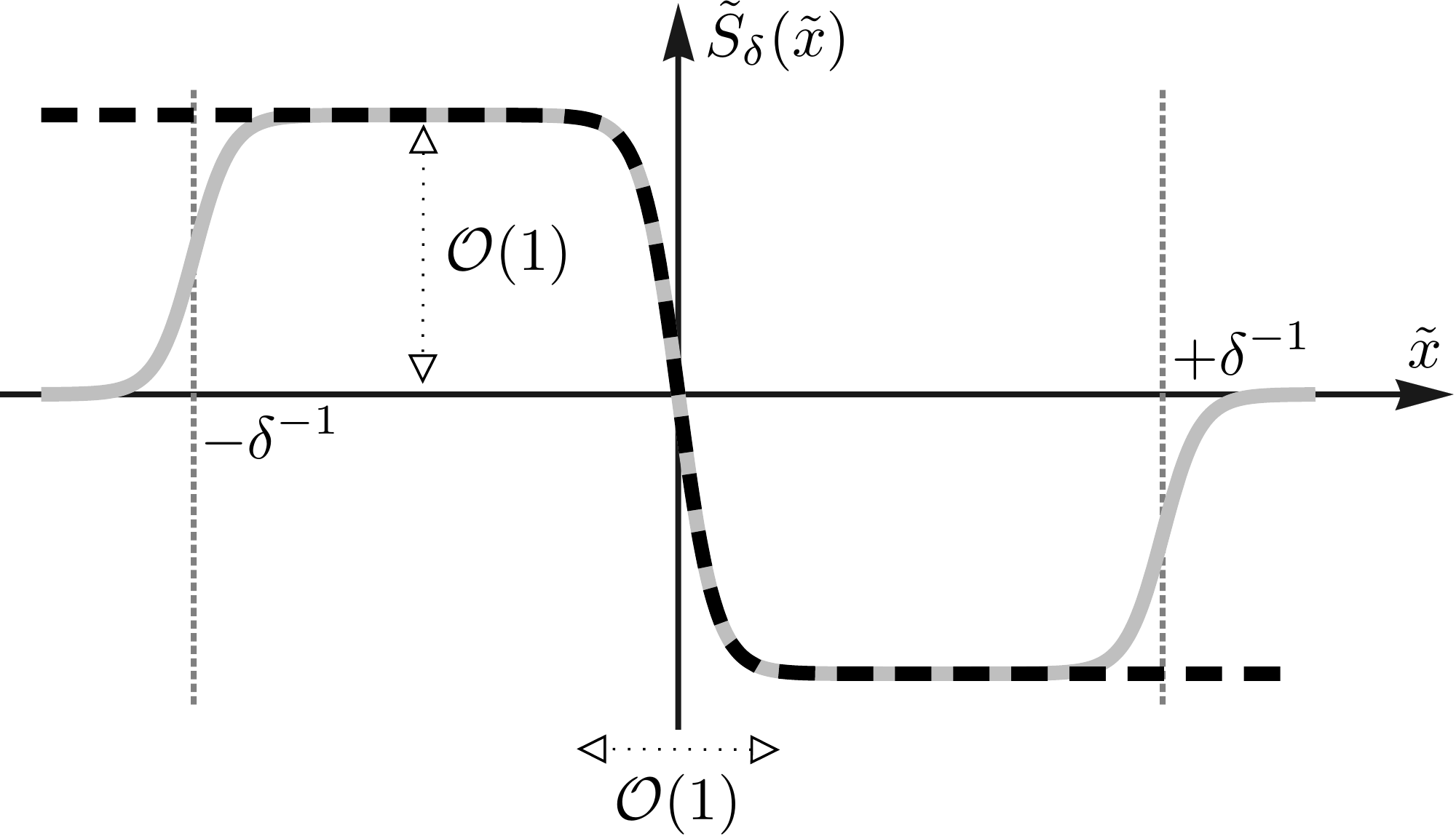} %
} %
\caption{Cartoon of the unique recaled eigenfunction $\tilde{S}_\delta$ (gray, solid) and its pointwise limit $\tilde{S}_0$ (black, dashed) with respect to the scaled phase variable $\tilde{x}$.%
} %
\label{FigEigen} %
\end{figure} %
{\bf Uniqueness of accumulation points:}   The linear ODE \eqref{AsympODE1} and the error estimates \eqref{eqn:errorEst} imply
\begin{align}
\label{AsympLaw1}
\tilde{S}_\delta\at{\tilde{x}}=c_{e,\delta}\tilde{T}_{e}
\at{\tilde{x}}+c_{o,\delta}\tilde{T}_o\at{\tilde{x}}+\DO{\delta^{m-1}}\qquad \text{for all}\quad \tilde{x}\in\tilde{I}_\delta
\end{align}
with $\tilde{T}_e$ and $\tilde{T}_o$ as in Proposition \ref{Prop.Coeff}. The constants $c_{e,\delta}$ and $c_{0,\delta}$ are uniquely determined by $\tilde{S}_\delta\at{0}$ and $\tilde{S}_\delta^\prime\at{0}$,  and satisfy
\begin{align*}
\abs{c_{e,\delta}}+\abs{c_{o,\delta}}\quad \xrightarrow{\;\;\delta\to0\;\;}\quad c\neq0
\end{align*}
due to the locally uniform convergence of $\tilde{S}_\delta$ and $\tilde{S}_\delta^\prime$ and the nontriviality of the limit. We further employ the identity
\begin{align*}
\tilde{S}_\delta\at{\tilde{x}}=
\tilde{S}_\delta\at{\tfrac32\delta^{-1}} +
\tilde{S}^\prime_\delta\bat{\tfrac32\delta^{-1}}\at{\tilde{x}-\tfrac32\delta^{-1}} +
\int\limits_{\tilde{x}}^{\tfrac32\delta^{-1}}
\tilde{S}_\delta^{\prime\prime}\at{\tilde{y}}\bat{\tilde{y}-\tilde{x}}\dint
{\tilde{y}}
\end{align*}
along with \eqref{Eqn:AuxEstimates} and the asymptotic differential relations \eqref{AsympODE1}+\eqref{AsympODE2} to get
\begin{align*}
\tilde{S}_\delta\at{\tfrac12\delta^{-1}}&=\int\limits_{\tfrac12\delta^{-1}}^{\tfrac32\delta^{-1}}
\tilde{S}_\delta^{\prime\prime}\at{\tilde{y}}\bat{\tilde{y}-\tfrac12\delta^{-1}}\dint
{\tilde{y}}+\DO{\delta^{m-1}}
\\&=
\int\limits_{-\tfrac12\delta^{-1}}^{+\tfrac12\delta^{-1}}
\tilde{S}_\delta^{\prime\prime}\at{\tilde{y}+\delta^{-1}}\bat{\tilde{y}+\tfrac12\delta^{-1}}\dint
{\tilde{y}}+\DO{\delta^{m-1}}
\\&=
\int\limits_{-\tfrac12\delta^{-1}}^{+\tfrac12\delta^{-1}}
\Bat{-\tfrac12 c_{e,\delta}\tilde{T}_e^{\prime\prime}\at{\tilde{y}}-\tfrac12 c_{o,\delta}\tilde{T}_o^{\prime\prime}\at{\tilde{y}}}\bat{\tilde{y}+\tfrac12\delta^{-1}}\dint
{\tilde{y}}+\DO{\delta^{m-1}}
\\&=
-\frac{c_{e,\delta}}{4\delta}\int\limits_{-\tfrac12\delta^{-1}}^{+\tfrac12\delta^{-1}}
\tilde{T}_e^{\prime\prime}\at{\tilde{y}}\dint{\tilde{y}}
-
\frac{c_{o,\delta}}{2}\int\limits_{-\tfrac12\delta^{-1}}^{+\tfrac12\delta^{-1}}
\tilde{T}_o^{\prime\prime}\at{\tilde{y}}\tilde{y}\dint
{\tilde{y}}+
\DO{\delta^{m-1}}
\\&=
-\frac{c_{e,\delta}}{2\delta}\tilde{T}_e^\prime\at{\tfrac12\delta^{-1}}
+c_{o,\delta}\tilde{T}_o\at{\tfrac12\delta^{-1}}+\DO{\delta^{m-1}},
\end{align*}
where we also used the parity of $\tilde{T}_e$ and $\tilde{T}_o$ as well as
\begin{align*}
\frac{\dint}{\dint \tilde{y}}\Bat{\tilde{T}_o^\prime\at{\tilde{y}}\tilde{y}-\tilde{T}_o\at{\tilde{y}}}=\tilde{T}_o^{\prime\prime}\at{\tilde{y}}\tilde{y},\qquad \tilde{T}^\prime_o\at{\tfrac12\delta^{-1}}=\DO{\delta^m}.
\end{align*}
Equating this with \eqref{AsympLaw1} evaluated at $\tilde{x}=\tfrac12\delta^{-1}$ we arrive at
\begin{align*}
\frac{c_{e,\delta}}{2\delta}\Bat{\tilde{T}_e^\prime\at{\tfrac12\delta^{-1}}+2\delta  \tilde{T}_e\at{\tfrac12\delta^{-1}}}=\DO{\delta^{m-1}}\,.
\end{align*}
On the other hand, the properties of $\tilde{T}_e$ -- see again  Proposition \ref{Prop.Coeff} -- provide
\begin{align*}
\tilde{T}_e^\prime\at{\tfrac12\delta^{-1}}=2\delta  \tilde{T}_e\at{\tfrac12\delta^{-1}}+\DO{\delta},\qquad
\tilde{T}_e^\prime\at{\tfrac12\delta^{-1}}=1+\DO{\delta^m}
\end{align*}
and we conclude that
\begin{align}
\label{AsympLaw2}
c_{e,\delta}=\DO{\delta^m}\,,\qquad c_{o,\delta}=c_{o,0}+\DO{\delta^m},
\end{align}
where $c_{o,0}\neq0$ is uniquely determined by the normalisation condition \eqref{eq:Gnormal}.

\bigpar
{\bf Conclusion:}  
In \eqref{AsympLaw1} and \eqref{AsympLaw2} have shown that $\tilde{S}_\delta$ can be approximated with high accuracy by a certain multiple of the odd solution to the linear ODE \eqref{eqn:linODE}, see Figure \ref{FigEigen} for an illustration, and Lemma \ref{lem:nabl} implies the corresponding asymptotic uniqueness for $\tilde{W}_\delta$. In particular, this result applies to the rescaled kernel functions $\pair{\tilde{S}_{1,\delta}}{\tilde{W}_{1,\delta}}$
from \eqref{eq:EV} as well as to the rescaling of any other solution to $L_\delta\pair{S_\delta}{W_\delta}=0$. If Proposi\-tion~\ref{prop:ker} was false, we would find another solution $\pair{\tilde{S}_\delta}{\tilde{W}_\delta}$ in the orthogonal $\fspaceL_a^2$-complement of $\pair{\tilde{S}_{1,\delta}}{\tilde{W}_{1,\delta}}$ and hence a contradiction.  

\subsection{Local uniqueness and differentiability of  travelling waves}

  We finally sketch the proof of Theorem \ref{ass:basic}.   We look for   solutions of the nonlinear travelling wave equation   \eqref{eq:F} in $\fspaceL^2_a$ and   thanks to    Lemma \ref{lem:nabl} we can recover $V$ for given $R$. So it   suffices to seek   solutions to the second order
nonlinear equation
\begin{equation}\label{eq:F2}
     \mathcal{F}_2(R,\si_\delta)=\frac{\partial^2}{(\partial x )^2}R(.)-  \frac 1 {\si_\delta} \Delta_1\Phi'\bat{R(.)} =0.\end{equation}
We note that $\mathcal{F}_2(.,\si_\delta): \fspaceH^2_a \to \fspaceL^2_a$ maps even to even and odd to odd functions and   aim to apply   the implicit function theorem to \eqref{eq:F2}. The solutions given in Proposition \ref{Ass:Waves} provide
a point with $\mathcal{F}_2(R_\delta,\si_\delta)=0$ and
the kernel of $L_\delta$ is   spanned by a single odd profile, see  Proposition \ref{prop:ker}.   By Lemma \ref{lem:nabl} b), $0$ is not in the essential \mbox{spectrum} and this implies that the second order version of $L_\delta$ as corresponding to \eqref{eq:2nd} is invertible on   the space of   even functions. Hence $D_1\mathcal{F}_2(R,\si_\delta)$ is invertible on even functions if $0<\delta<\delta_0$. Consequently, the uniqueness part of Theorem \ref{ass:basic} is a consequence of the implicit function theorem. Furthermore, $R$ depends smoothly on the wave speed parameter $\sigma$ as long as $\delta$ is small enough such that $\sigma$ will be large. This completes the proof
 of Theorem \ref{ass:basic}.

\section{Discussion}
\label{sect:dis}

The control of the kernel of $L_\delta$ is an important step to study the dynamical \mbox{stability} of the waves given in Proposition \ref{Ass:Waves}. Following \cite{FP3} it is enough to study eigenfunctions to eigenvalues with non-negative real part of the linearisation of \eqref{Eqn:FPU} around the travelling waves. The current analysis helps with this as one needs to show that neutral modes are
just those
 $2\times 2$ Jordan blocks that are created due to the symmetry of  the system. The symmetry solutions are $\bpair{S_{1,\delta}}{W_{1,\delta}}$ from \eqref{eq:EV} and
\begin{align*}
 \bpair{S_{2,\delta}}{W_{2,\delta}}:=\pair{\frac{d R_\delta}{d\delta}}{\frac{d V_\delta}{d\delta}}
\end{align*}
and satisfy the Jordan relations
\begin{align*}
L_\delta   \bpair{S_{1,\delta}}{W_{1,\delta}}=0 ,\qquad
L_\delta  \bpair{S_{2,\delta}}{W_{2,\delta}}=-\frac{d\sqrt{\si_\delta}}{d\delta} \bpair{S_{1,\delta}}{W_{1,\delta}}.
\end{align*}
This programme will be carried out in forth-coming paper for the high-energy limit using a similar combination of techniques of detailed asymptotic analysis and the structure of the underlying equations. Most of the analysis will hold for other \mbox{potentials} than \eqref{Eqn:Pot} as long as one can guarantee certain  non-degeneracy conditions for the energy of   a solitary wave.   In particular, one needs to show   that
\[\frac{dH(R_\delta,V_\delta)}{d\delta} \neq 0\quad \mbox{ and }\quad \frac{d\sigma_\delta}{d\delta} \neq 0\]
  holds in the high-energy limit, where $H$ can be computed using the FPU energy. 

Unimodal solitary travelling waves exist following \cite{FW96} for all supersonic wave speeds. They are locally unique and dynamically stable in KdV regime close to the sound speed by \cite{FP99,FP02,FP3,FP4}. For the high-energy, i.e. high velocity limit,  we have established local uniqueness in this note, whereas  results on dynamical stability are forthcoming. We conjecture that for most potentials the whole family of unimodal solitary travelling waves are indeed unique and stable, but new methods need to be developed to understand the linearisation of \eqref{Eqn:FPU} around the travelling waves for moderate speeds.

\section*{Acknowledgements}
The authors are grateful for the support by the \emph{Deutsche Forschungsgemeinschaft} (DFG individual grant HE 6853/2-1) and the \emph{London Mathematical Society} (LMS Scheme 4 Grant, Ref~41326).
KM would like to thank for the hospitality during a sabbatical stay at the University of M\"unster.
%
\bibliographystyle{alpha}
\newcommand{\etalchar}[1]{$^{#1}$}

\end{document}